\documentclass[10pt, english, draft]{article}
\usepackage{amssymb,amsmath}
\title{\bf The number of spanning trees of
power graphs associated with specific groups and some
applications}
\author{{\bf A. R. Moghaddamfar} and {\bf S. Rahbariyan}\\[0.1cm]
{\em Faculty of Mathematics,}\\[0.1cm] {\em K. N. Toosi
University of Technology,}\\[0.1cm]
 {\em P. O. Box $16315$-$1618$, Tehran, Iran}\\[0.1cm]
{\em E-mails}:  {\tt moghadam@kntu.ac.ir} \ {\em and} \  {\tt
moghadam@ipm.ir}\\[0.2cm]
{\bf S. Navid Salehy} and {\bf S. Nima Salehy}\\[0.1cm]
{\em Department of Mathematics, Florida State University,}\\[0.1cm]
{\em Tallahassee, FL $32306$, USA.}\\[0.1cm]
{\em E-mails}:  {\tt navidsalehy@math.fsu.edu}\\ {\em and} \ \
{\tt
nimasalehy@math.fsu.edu}\\[0.2cm]
{\em In memory of Professor Michael Neumann.}}

\newenvironment{proof}{\noindent {\em {Proof}}.}{$\square$
\medskip}

\newtheorem{corollary}{Corollary}[section]

\newtheorem{theorem}{Theorem}[section]
\newtheorem{proposition}{Proposition}[section]
\newtheorem{lm}{Lemma}[section]

\newtheorem{qu}{Question}[section]
\begin{document}
\maketitle
\begin{abstract}
\noindent Given a group $G$, we define the power graph
$\mathcal{P}(G)$ as follows: the vertices are the elements of $G$
and two vertices $x$ and $y$ are joined by an edge if $\langle
x\rangle\subseteq \langle y\rangle$ or $\langle y\rangle\subseteq
\langle x\rangle$. Obviously the power graph of any group is
always connected, because the identity element of the group is
adjacent to all other vertices. In the present paper, among other
results, we will find the number of spanning trees of the power
graph associated with specific finite groups. We also determine,
up to isomorphism, the structure of a finite group $G$ whose power
graph has exactly $n$ spanning trees, for $n<5^3$. Finally, we
show that the alternating group $\mathbb{A}_5$ is uniquely
determined by tree-number of its power graph among all finite
simple groups.
\end{abstract}
\def\thefootnote{ \ }
\footnotetext{{\small {\it $2010$ Mathematics Subject
Classification}: 05C05, 05C25, 05C30, 05C50.
\\[0.1cm] {\em Keywords}: power graph, group, tree-number. }}

\renewcommand{\baselinestretch}{1}
\def\thefootnote{ \ }
\section{Problem Statement and Motivation}
Throughout this paper, only finite groups will be considered.
Moreover, all the graphs under consideration are finite, simple
(with no loops or multiple edges) and undirected. Given a
connected graph $\Gamma$ with $n$ vertices, a {\em spanning tree}
of $\Gamma$ is a connected subgraph $T$ of $\Gamma$ which has
$n-1$ edges. Spanning trees of connected graphs have been the
focus of considerable research. Actually, one of the interesting
problems in Graph Theory is the problem of finding the number of
spanning trees of a connected graph (also called the {\em
complexity} of $\Gamma$, see \cite{Biggs}), which arises in a
variety of applications. Especially, it is of interest in the
analysis of electrical networks.

There are scattered results in the literature finding an explicit
simple formula for the number of spanning trees of special graphs.
Nevertheless, the problem becomes more interesting when we are
dealing with a graph which is mainly associated with an algebraic
structure, for instance, a group. Given a finite group $G$, there
are many different ways to associate a simple graph $\Gamma_G$ to
$G$ by choosing families of its elements or subgroups as vertices
and letting two vertices be joined by an edge if and only if they
satisfy a certain relation. Now, one of the interesting questions
is to ask about characterizing certain properties of the group by
means of the properties of the associated graph, in other words,
to study the influence of a property of the graph on the
structure of the group. This line of research has attracted
considerable attention in recent years (see, for instance,
\cite{abe, dolfi, Mszz, ss}).

A graph which has recently deserved a lot of attention is the
{\em power graph} associated with a group $G$ (see \cite{n11,
CSS, KQ1, KQ}). Note that, the term ``power graph" was introduced
and first considered in \cite{KQ1}. In this graph, the vertices
are all elements of the group $G$ and two different vertices $x$
and $y$ are adjacent, and we write $x\sim y$, when $\langle
x\rangle\subseteq \langle y\rangle$ or $\langle y\rangle\subseteq
\langle x\rangle$. It is evident from the definition that the
power graph of any group is always {\em connected}, because the
identity element of the group is adjacent to all other vertices.
We denote by $\kappa(G)$ the number of spanning trees of the
power graph $\mathcal{P}(G)$ of a group $G$ and call this number
the tree-number of $\mathcal{P}(G)$, which will be investigated
for certain finite groups in this paper. More precisely, we will
obtain the explicit formulas for the number of spanning trees of
power graphs associated with the {\em cyclic group}
$\mathbb{Z}_n$, {\em dihedral group} $D_{2n}$ and the {\em
generalized quaternion group} $Q_{4n}$.

{\em Remark $1$.} These groups are considered as they play an
important role in some of the deeper parts of Finite Group Theory
and in most cases they appear as subgroups of a given group. For
example, the Cauchy's theorem \cite[Theorem 5.11]{rose} states
that if $G$ is a finite group and $p$ is a prime divisor of
$|G|$, then $G$ has at least one cyclic subgroup of order $p$.

{\em Remark $2$.} The main tool for computing these explicit
formulas is a well-known theorem due to Temperley (Theorem
\ref{th-1}), which deals with the explicit computation of a
determinant. Note that, many useful and efficient tools for
evaluating determinants are provided in \cite{K2}. It is worth
stating at this point that although we are mainly trying to
obtain these explicit formulas for the number of spanning trees of
specific graphs, meanwhile we will be faced with some {\em
interesting integer matrices} whose determinants are needed.

Evidently for two isomorphic groups $G$ and $H$ we have
$\kappa(G)=\kappa(H)$. Dose the converse hold? The answer to this
question is negative in the general case. Actually, $\kappa(G)$
can not determine the structure of a group $G$ uniquely. There
are many examples which justify this matter. For instance, for
all finite elementary abelian $2$-groups $G$ we have
$\kappa(G)=1$, or as another example
$\kappa(\mathbb{Z}_3)=\kappa(\mathbb{S}_3)=3$, where
$\mathbb{S}_3$ denotes the symmetric group on 3 letters.

A group $G$ from a class $\mathcal{C}$ is said to be recognizable
in $\mathcal{C}$ by $\kappa(G)$ (shortly, $\kappa$-recognizable in
$\mathcal{C}$) if every group $H \in \mathcal{C}$ with $\kappa
(H)=\kappa (G)$ is isomorphic to $G$. In other words, $G$ is
$\kappa$-recognizable in $\mathcal{C}$ if $h_{\mathcal{C}}(G)=1$,
where $h_{\mathcal{C}}(G)$ is the (possibly infinite) number of
pairwise non-isomorphic groups $H\in \mathcal{C}$ with
$\kappa(H)=\kappa(G)$. We denote by $\mathcal{F}$ and
$\mathcal{S}$ the classes of all finite groups and all finite
simple groups, respectively. There are some examples of groups
with $1<h_{\mathcal{F}}(G)<\infty$. For instance,
$h_{\mathcal{F}}(\mathbb{Z}_3)=2$ and
$\kappa(\mathbb{Z}_3)=\kappa(\mathbb{S}_3)=3$ (see Corollary
\ref{coro-small}). In the present paper, we find the first
example of $\kappa$-recognizable group in class $\mathcal{S}$. It
turns out that the following is true:
\begin{theorem}\label{th1-new} The alternating group $\mathbb{A}_5$ is
$\kappa$-recognizable in the class of all finite simple groups,
that is, $h_{\mathcal{S}}(\mathbb{A}_5)=1$.
\end{theorem}

We will also continue to ask the following two questions:

\begin{qu} Is a group $G$ isomorphic to $\mathbb{A}_n$ ($n\geqslant 4$) if and only if
$\kappa(G)=\kappa(\mathbb{A}_n)$?
\end{qu}

\begin{qu} Given a natural number $n$, determine
all groups $G$ whose power graph has exactly $n$ spanning trees,
that is $\kappa(G)=n$.
\end{qu}

{\em Remark $3$.} Note that
$\kappa(\mathbb{A}_4)=\kappa(\mathbb{Z}_3\times \mathbb{Z}_3)=3^4$
(see Table 1). Moreover, there are some natural numbers $n$ for
which a group $G$ does not exist with $\kappa(G)=n$. For example,
there does not exist a group $G$ with $\kappa(G)=2$.

After this introduction, the structure of this paper is organized
as follows: basic definitions and notation are summarized in
Section 2. In Section 3, we derive some preparatory results. In
Sections 4 and 5, we obtain explicit formulas for the tree-number
of power graphs associated with a cyclic group $\mathbb{Z}_n$, a
dihedral group $D_{2n}$ and a generalized quaternion group
$Q_{4n}$. Finally, a few applications of obtained results are
presented in Section 6: (1) a classification of groups $G$ for
which $\kappa(G)<5^3$, (2) a new chracterization of the
alternating group $\mathbb{A}_5$ by $\kappa(\mathbb{A}_5)$, and
(3) a list of $\kappa(G)$ for all groups $G$ with $|G|\leqslant
15$.
\section{Basic Definitions and Notation}
The notation and definitions used in this paper are standard and
taken mainly from \cite{Biggs, rose, west}. We will cite only a
few. Let $\Gamma=(V, E)$ be a simple graph. We denote by
$\mathbf{A}=\mathbf{A}(\Gamma)$ the adjacency matrix of $\Gamma$.
The complement $\overline{\Gamma}$ of $\Gamma$ is the simple
graph whose vertex set is $V$ and whose edges are the pairs of
non-adjacent vertices of $\Gamma$. When $U\subseteq V$, the
induced subgraph $\Gamma[U]$ is the subgraph of $\Gamma$ whose
vertex set is $U$ and whose edges are precisely the edges of
$\Gamma$ which have both ends in $U$. Two graphs are disjoint if
they have no vertex in common, and edge-disjoint if they have no
edge in common. If $\Gamma_1$ and $\Gamma_2$ are disjoint, we
refer to their union as a disjoint union, and generally denote it
by $\Gamma_1\oplus\Gamma_2$. By starting with a disjoint union of two
graphs $\Gamma_1$ and $\Gamma_2$ and adding edges joining every
vertex of $\Gamma_1$ to every vertex of $\Gamma_2$, one obtains
the join of $\Gamma_1$ and $\Gamma_2$, denoted $\Gamma_1\vee
\Gamma_2$. A clique in a graph is a set of pairwise adjacent
vertices. The clique number of a graph $\Gamma$, written
$\omega(\Gamma)$, is the number of vertices in a maximum clique
of $\Gamma$. Given a group $G$, we denote by $\omega(G)$ the set
of orders of all elements in a group $G$ and call this set the
spectrum of $G$. The spectrum $\omega(G)$ of $G$ is closed under
divisibility and determined uniquely from the set $\mu(G)$ of
those elements in $\omega (G)$ that are maximal under the
divisibility relation. In the case when $\mu(G)$ is a one-element
set $\{n\}$, we write $\mu(G)=n$. For a natural number $m$, the
alternating and symmetric group of degree $m$ denoted by
$\mathbb{A}_m$ and $\mathbb{S}_m$, respectively.
\section{Auxiliary Results} In this section we give several auxiliary
results to be used later. The first of them is the following
lemma (See Lemma 3.4 and Corollary 3.1 in \cite{banoo1}).
\begin{lm}\label{pre-1} Let $G=\langle x\rangle$ be a cyclic
group of order $n$ and $\Gamma=\mathcal{P}(G)$. Then the degree
of $x^m\in G$ in the power graph $\Gamma$ is given by
$$d_\Gamma(x^m)=\frac{n}{(m,
n)}-1+\sum_{d|(m,n) \atop d\neq
(m,n)}\phi\left(\frac{n}{d}\right),$$ where $\phi(k)$ signifies
the Euler function of a natural number $k$. In particular, all
non-trivial elements of a cyclic group with the same orders have
the same degrees in its power graph.
\end{lm}

A complete graph is a simple graph in which any two vertices are
adjacent. The complete graph on $n$ vertices is denoted by
$K_n$.  Next lemma is taken from \cite{CSS}.
\begin{lm}\label{complete-CSS}{\rm \cite[Theorem 2.12]{CSS}} \
Let $G$ be a finite group. Then  $\mathcal{P}(G)$ is complete if
and only if $G$ is a cyclic group of order $1$ or $p^m$ for some
prime number $p$ and for some natural number $m$.
\end{lm}

As already mentioned in the Introduction, the number of spanning
trees of a graph is one of the most important graph-theoretical
parameters and appears in a number of applications. Given a graph
$\Gamma$, we denote by $\kappa(\Gamma)$, the number of spanning
trees of a graph $\Gamma$. In \cite{Biggs}, Biggs refers to
$\kappa(\Gamma)$ as the {\em tree-number} of $\Gamma$. Clearly
$\kappa(\Gamma)=0$ if and only if $\Gamma$ is disconnected. In
\cite{Cayley}, Cayley devised the well-known formula
$\kappa(K_n)=n^{n-2}$. The Laplacian matrix $\mathbf{Q}$ of a
graph $\Gamma$ is $\Delta-\mathbf{A}$, where $\Delta$ is the
diagonal matrix whose $i$-th diagonal entry is the degree $v_i$
in $\Gamma$ and $\mathbf{A}$ is the adjacency matrix of $\Gamma$.
The following Theorem is due to Temperley (1964).
\begin{theorem}[\cite{Tem}]\label{th-1}
The number of spanning trees of a graph $\Gamma$ with $n$ vertices
is given by the formula $$\kappa(\Gamma)=\det
(\mathbf{J}+\mathbf{Q})/n^2,$$ where $\mathbf{J}$ denotes the
matrix each of whose entries is $+1$.
\end{theorem}

Given a graph $\Gamma$, we will let $c(\Gamma)$ denote the number
of connected components of $\Gamma$. A cut edge of $\Gamma$ is an
edge $e$ such that $c(\Gamma-e)>c(\Gamma)$. Similarly, a cut
vertex of $\Gamma$ is a vertex $v$ such that
$c(\Gamma-v)>c(\Gamma)$. In particular, a cut edge (resp. a cut
vertex) of a connected graph is an edge (resp. a vertex) whose
deletion results in a disconnected graph. For any edge $e$ which
is not a loop, we also define the graph $\Gamma \cdot e$ to be
the subgraph obtained from $\Gamma-e$ by identifying the vertices
of $e$. The following is well known, see for example
\cite[Proposition 2.2.8]{west}.
\begin{theorem}[Deletion-Contraction Theorem]\label{th-2}
The number of spanning trees of a graph $\Gamma$ satisfies the
deletion-contraction recurrence
$$\kappa(\Gamma)=\kappa(\Gamma-e)+\kappa(\Gamma\cdot e),$$ where
$e\in E(\Gamma)$. In particular, if $e\in E(\Gamma)$ is a
cut-edge, then $$\kappa(\Gamma)=\kappa(\Gamma\cdot e).$$
\end{theorem}
The following theorem follows in a straightforward way from the
multiplication principle in combinatorics.
\begin{theorem}\label{th-222} Let $\Gamma$ be a connected graph
and let $v$ be a cut vertex of $\Gamma$ with
$$\Gamma-v=\Gamma_1\oplus\Gamma_2\oplus\cdots\oplus\Gamma_{c},$$
where $\Gamma_i$, $i=1,2,\ldots, c$, is the $i$th connected
component of $\Gamma-v$ and $c=c(\Gamma-v)$. Set
$\widetilde{\Gamma}_i=\Gamma_i+v$. Then, there holds
$$\kappa(\Gamma)=\kappa(\widetilde{\Gamma}_1)
\times\kappa(\widetilde{\Gamma}_2)\times \cdots\times
\kappa(\widetilde{\Gamma}_c).$$
\end{theorem}

In what follows, we denote by $\pi(n)$ the set of the prime
divisors of a positive integer $n$. Given a group $G$, we will
write $\pi(G)$ instead of $\pi(|G|)$. If $p\in \pi(G)$, then
${\rm Syl}_p(G)$ will denote the set of all Sylow $p$-subgroups
of $G$.
\begin{theorem} \label{th-4}
Let $H_1, H_2, \ldots, H_t$ be nontrivial subgroups of a group $G$
such that $$H_i\cap H_j=\{1\}, \ \ \ \mbox{for each} \ \
1\leqslant i<j\leqslant t. $$ Then, there hold:
\begin{itemize}
\item[$(a)$]
$\kappa(G)\geqslant\kappa(H_1)\kappa(H_2)\cdots\kappa(H_t)$. In
particular, if $\pi(G)=\{p_1, \ldots, p_k\}$ and $\mu(P_i)=m_i$,
where $P_i\in {\rm Syl}_{p_i}(G)$, $1\leqslant i\leqslant k$, then
$$\kappa(G)\geqslant \prod_{i=1}^{k}m_i^{m_i-2}\geqslant \prod_{i=1}^{k}p_i^{p_i-2}.$$
\item[$(b)$] If $G=H_1\cup
H_2\cup\cdots\cup H_t$, then we have
$$\kappa(G)=\kappa(H_1) \kappa(H_2) \cdots \kappa(H_t).$$
\end{itemize}
\end{theorem}
\begin{proof} $(a)$ Recall that by Proposition 4.5 in \cite{CSS}, $\mathcal{P}(H_i)$ is an
induced subgraph of $\mathcal{P}(G)$. Let
$H:=\bigcup\limits_{i=1}^tH_i$. Now to each spanning tree
$\bigcup\limits_{i=1}^tT_{H_i}$ of
$\bigcup\limits_{i=1}^t\mathcal{P}(H_i)$,  we associate a
spanning tree
$$T_G=\bigcup_{i=1}^tT_{H_i}\cup \bigcup_{g\in G\setminus H}\{1, g\},$$
of  $\mathcal{P}(G)$, which shows that
$\kappa(G)\geqslant\kappa(H_1)\kappa(H_2)\cdots\kappa(H_t)$.

Let $x_i$ be a $p_i$-element of $G$ of order $m_i$ and
$Q_i:=\langle x_i\rangle$. Then, by Lemma \ref{complete-CSS},
$\mathcal{P}(Q_i)$ is a complete graph of order $m_i$, and so
Cayley formula implies that $\kappa(Q_i)=m_i^{m_i-2}$. The result
now follows by applying the first part.

$(b)$  It is a straightforward verification.
\end{proof}

As immediate consequences of Theorem \ref{th-4}, we have the
following four corollaries.

\begin{corollary}\label{cor-dir} Let $G=H_1\times H_2\times \cdots\times H_n$, where
$n$ is a positive integer. Then $$\kappa(G)\geqslant
\kappa(H_1)\kappa(H_2)\cdots\kappa(H_n).$$
\end{corollary}

\begin{corollary}\label{cor-semidir} If $G=K\rtimes C$ is a semidirect product
of $K$ by $C$ (especially, if $G$ is a Frobenius group with
kernel $K$ and complement $C$), then $$\kappa(G)\geqslant
\kappa(K)\kappa(C).$$
\end{corollary}

\begin{corollary}\label{cor-semidir-1} Let $G$ be a finite group and let $p$ be the smallest prime such that
$\kappa(G)<p^{p-2}$. Then $\pi(G)\subseteq \pi((p-1)!)$.
\end{corollary}

Given a group $G$, we put $G^\#=G\setminus \{1\}$. A group $G$ is called {\em Element Prime Order} group if every nonidentity element of $G$ has prime order, i.e.,  $\omega(G)\setminus \{1\}
=\pi(G)$. We can consider an EPO-group $G$ as follows:
$$G=\biguplus_{p\in \pi(G)}(\underbrace{\mathbb{Z}_p^{\#}\uplus \cdots \uplus \mathbb{Z}_p^{\#}}_{c_p\mbox{-times}})\cup \{1\},$$
where $c_p$ signifies the number of cyclic subgroups of order $p$
in $G$, and hence
$$\mathcal{P}(G)=K_1\vee
\bigoplus_{p\in \pi(G)} c_pK_{p-1}.$$

\begin{corollary}\label{cor-epo} Let $G$ be an EPO-group. Then,
there holds
$$\kappa(G)=\prod_{p\in \pi(G)}p^{(p-2)c_p}.$$
\end{corollary}

In the same manner as in the proof of Theorem \ref{th-4}, we can
prove the following theorem:

\begin{theorem} \label{th-5}
Let $G$ be a group and let $\Omega_1, \Omega_2, \ldots, \Omega_t$
be cliques of $\mathcal{P}(G)$ such that $\Omega_i\cap
\Omega_j=\{1\}$, for each $1\leqslant i<j\leqslant t$.  Then,
there holds
$$\kappa(G)\geqslant \prod_{i=1}^{t}|\Omega_i|^{|\Omega_i|-2}\geqslant \omega^{\omega-2},$$
where $\omega=\omega(\mathcal{P}(G))$. In particular, we have
$$\kappa(G)\geqslant \prod_{p\in \pi(G)}p^{(p-2)c_p},$$ where $c_p$
signifies the number of cyclic subgroups of order $p$ in $G$.
\end{theorem}
{\em Some Examples.} $(a)$ The power graph $\mathcal{P}(G)$ of an
elementary $p$-group $G$ of order $p^n$ consist of
$$(p^n-1)/(p-1)=p^{n-1}+p^{n-2}+\cdots+p+1,$$ cliques on $p$ vertices sharing
the identity element. Now, using Corollary \ref{cor-epo}, we get
\begin{equation}\label{equation-1}
\kappa(G)=p^{\frac{p^n-1}{p-1}(p-2)}.
\end{equation}
$(b)$ If $G$ is a non-abelian group of order 21, then $G\cong
\mathbb{Z}_7\rtimes \mathbb{Z}_3$. Therefore, using Corollary
\ref{cor-semidir}, $\kappa(G)\geqslant \kappa(\mathbb{Z}_7)\cdot
\kappa(\mathbb{Z}_3)=7^5\cdot 3$. However, $G$ is an EPO-group
with $\omega(G)=\{1, 3, 7\}$, and we have
$$\mathcal{P}(G)=K_1\vee(7K_2\oplus K_6),$$
for which we conclude that $\kappa(G)=3^7\cdot 7^5$.

$(c)$ According to Theorem 11.3 in \cite{rose},  a group is
nilpotent if and only if it is the direct product of its Sylow
subgroups. Hence, if $G$ is a nilpotent group, then
$$G\cong \prod_{p\in \pi(G)}G_p,$$ where $G_p\in {\rm Syl}_p(G)$.
Now, by Corollary \ref{cor-dir}, we obtain $$\kappa(G)\geqslant
\prod _{p\in \pi(G)}\kappa(G_p).$$ In particular, if
$n=p_1^{\alpha_1}p_2^{\alpha_2}\cdots p_k^{\alpha_k}$ is the
prime factorization of a natural number $n$, where $k, \alpha_1,
\ldots, \alpha_k$ are positive integers and $p_1, \ldots, p_k$
are distinct primes, then $\mathbb{Z}_n\cong
\mathbb{Z}_{p_1^{\alpha_1}}\times \mathbb{Z}_{p_2^{\alpha_2}}
\times \cdots \times \mathbb{Z}_{p_k^{\alpha_k}}$, and again, by
Corollary \ref{cor-dir} and Cayley formula, we have:
$$\kappa(\mathbb{Z}_n)\geqslant \prod_{i=1}^{k}\kappa(\mathbb{Z}_{p_i^{\alpha_i}})
=\prod_{i=1}^{k}(p_i^{\alpha_i})^{p_i^{\alpha_i}-2}.$$ In the next
section, we will find an explicit formula for the tree-number
$\kappa(\mathbb{Z}_n)$.
\section{The Tree-number of $\mathcal{P}(\mathbb{Z}_n)$}
Let $n$ be a natural number and let $G=\mathbb{Z}_n$ be a finite
cyclic group of order $n$. Then for every divisor $d$ of $n$,
there exists a unique subgroup of order $d$, and so the number of
all elements of order $d$ in $G$ is equal to $\phi(d)$.
Therefore, if $d_1>d_2>\cdots>d_k$ are all divisors of $n$
(evidently $d_1=n$ and $d_k=1$), then we
have:$^\ast$\footnote{$^\ast$It is well known that $n=\sum_{d|n}
\phi(d)$.}
$$n=\sum_{i=1}^{k}\phi(d_i).$$
As we mentioned in Lemma \ref{pre-1}, all non-trivial elements of
$G$ with the same orders have the same degrees in
$\mathcal{P}(G)$. Let $n_i$, $i=1, 2, \ldots, k$, denote the
degree of all elements of order $d_i$ in $\mathcal{P}(G)$. Let
$\mathbf{A}$ be the adjacency matrix of $\mathcal{P}(G)$, and let
$\mathbf{Q}$ be the Laplacian matrix. Then the matrix
$\mathbf{J}+\mathbf{Q}$ has the following block-matrix structure:
\begin{equation}\label{e2} \mathbf{J}+\mathbf{Q}=\left[
\begin{array}{cccc}
D_{11} & D_{12} & \ldots & D_{1k} \\[0.1cm]
D_{21} & D_{22} & \ldots & D_{2k} \\[0.1cm]
\vdots & \vdots & \ddots & \vdots \\[0.1cm]
D_{k1} & D_{k2} & \ldots & D_{kk} \\[0.1cm]
\end{array}\right ]
\end{equation} where $D_{ij}$ is a matrix of size $\phi(d_i)\times \phi(d_j)$
with
$$D_{ij}=\left\{\begin{array}{lll} m_i\mathbf{I} & if & i=j,\\[0.2cm]
0 & if & i\neq j, \ \ d_i|d_j \  \mbox{or} \ d_j|d_i,\\[0.2cm]
\mathbf{J} &  & \mbox{otherwise},
\end{array} \right.$$
where $m_i=n_i+1$, $i=1, 2, \ldots, k$. For instance, in the case
when $G=\mathbb{Z}_{12}$, the matrix $\mathbf{J}+\mathbf{Q}$ is
given by:
\begin{equation}\label{e3}
\left[ \begin{array}{llll|ll|ll|ll|l|l}
m_1 & . & .& .& .& .&  .& .& .& .& .& . \\[0.1cm]
. & m_1 & .& .& .& .&  .& .& .& .& .& . \\[0.1cm]
. & . & m_1& .& .& .&  .& .& .& .& .& . \\[0.1cm]
. & . & .& m_1& .& .&  .& .& .& .& .& . \\[0.1cm]
\hline
. & . & .& .& m_2& .&  1& 1& .& .& .& . \\[0.1cm]
. & . & .& .& .& m_2&  1& 1& .& .& .& . \\[0.1cm]
\hline
. & . & .& .& 1& 1&  m_3& .& 1& 1& .& . \\[0.1cm]
. & . & .& .& 1& 1&  .& m_3& 1& 1& .& . \\[0.1cm]
\hline
. & . & .& .& .& .&  1 & 1& m_4& .& 1& . \\[0.1cm]
. & . & .& .& .& .&  1 & 1& .& m_4& 1& . \\[0.1cm]
\hline
. & . & .& .& .& .&  . & .& 1& 1& m_5& . \\[0.1cm]
\hline
. & . & .& .& .& .&  . & .& .& .& .& m_6 \\[0.1cm]
\end{array}\right ].\end{equation}
In this section, we set
$$\lambda_i:=\frac{m_i}{\phi(d_i)}, \ i=2, \ldots, k-1, \ \ \mbox{and} \ \ \Phi:=\prod_{i=2}^{k-1}\lambda_i.$$

To state our first result, we have to introduce a new definition.
Actually, the {\em divisor graph} $D(n)$ of a natural number $n$
is defined as follows: the vertex set of this graph is
$\pi_d(n)$, the set of all divisors of $n$, and two divisors
$d_i$ and $d_j$ of $n$ are adjacent if and only if $d_i|d_j$ or
$d_j|d_i$. As usual, we denote the complement of $D(n)$ by
$\overline{D}(n)$. For instance, the subgraph $D(30)\backslash
\{1, 30\}$ of divisor graph $D(30)$ and its complement are
depicted in Fig. 1.

\vspace{1.5cm} \setlength{\unitlength}{3mm}
\begin{picture}(0,0)(-7,9)
\linethickness{0.7pt} %
\put(0,10){\circle*{0.3}}%
\put(6,10){\circle*{0.3}}%
\put(0,7){\circle*{0.3}}%
\put(6,7){\circle*{0.3}}%
\put(3,11.5){\circle*{0.3}}%
\put(3,5.5){\circle*{0.3}}%
\put(0,10){\line(2,1){3}}%
\put(0,10){\line(0,-1){3}}%
\put(0,7){\line(2,-1){3}}%
\put(3,5.5){\line(2,1){3}}%
\put(6,7){\line(0,1){3}}%
\put(6,10){\line(-2,1){3}}%
\put(-0.9,10.3){$2$}%
\put(6.5,10.3){$5$}%
\put(-0.9,6.5){$6$}%
\put(2.7,4.2){$3$}%
\put(6.5,6.5){$15$}%
\put(2.6,12){$10$}%
\put(12.2,10.3){$2$}%
\put(19.2,10.3){$5$}%
\put(12.1,6.5){$6$}%
\put(15.7,4.2){$3$}%
\put(19.4,6.5){$15$}%
\put(15.6,12){$10$}%
\put(13,10){\circle*{0.3}}%
\put(19,10){\circle*{0.3}}%
\put(13,7){\circle*{0.3}}%
\put(19,7){\circle*{0.3}}%
\put(16,11.5){\circle*{0.3}}%
\put(16,5.5){\circle*{0.3}}%
\put(13,10){\line(2,-1){6}}%
\put(13,7){\line(2,1){6}}%
\put(13,7){\line(2,3){3}}%
\put(19,7){\line(-2,3){3}}%
\put(13,10){\line(2,-3){3}}%
\put(19,10){\line(-2,-3){3}}%
\put(13,7){\line(1,0){6}}%
\put(13,10){\line(1,0){6}}%
\put(16,5.5){\line(0,0){6}}%
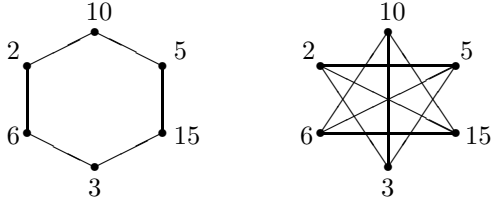
\put(-5,2){{\small {\bf Fig. 1.} The graph $D(30)\backslash \{1,
30\}$ and its complement $\overline{D}(30)\backslash \{1, 30\}$.}}
\end{picture}\\[2.2cm]

We are now ready to calculate the number of spanning trees of a
power graph associated with a cyclic group of order $n$.
\begin{theorem}\label{result2} Let $n$ be a positive integer and $d_1>d_2>\cdots>d_k$ all divisors of
$n$. With the notation as explained above, there holds
\begin{equation}\label{trees-cyclic} \kappa(\mathbb{Z}_{n})=\prod_{i=1}^{k}m_i^{\phi(d_i)}\Big(\Phi+
\sum_{\Lambda} \det\mathbf{A}(\Lambda)\lambda_2^{t_2}
\lambda_3^{t_3}\cdots \lambda_{k-1}^{t_{k-1}}\Big)/(\Phi
n^2),\end{equation} where $t_i\in \{0, 1\}$, $i=2, 3, \ldots,
k-1$, and the summation is over all induced subgraphs $\Lambda$ of
$\overline{D}(n)\setminus \{d_1, d_k\}$ to the vertices
$\{d_{i_1}, \ldots, d_{i_s}\}$ corresponding to
$t_{i_1}=\cdots=t_{i_s}=0$.
\end{theorem}
{\em Remark $4$.} At first glance, it seems that Eq.
(\ref{trees-cyclic}) is a very complicated formula. But, it is
worth mentioning that in the right-hand side of Eq.
(\ref{trees-cyclic}), the orders of determinants $\det
\mathbf{A}(\Lambda)$ are at most $k-2$, which is pretty small
compared with the order of determinant
$\det(\mathbf{J}+\mathbf{Q})$ needed to compute the tree-number
$\kappa(\mathbb{Z}_{n})$. After the proof of this theorem, we
will present some examples to illustrate the effectiveness of
this formula.

\begin{proof} First, we illustrate the proof for a special
case and then describe the necessary modifications in the general
case. For instance, we consider the cyclic group
$G=\mathbb{Z}_{12}$. In this situation, we have $$n=12, \ d_1=12,
\ d_2=6, \  d_3=4, \  d_4=3, \  d_5=2, \ d_6=1.$$ Then the matrix
$\mathbf{J}+\mathbf{Q}$ has the form as Eq. (\ref{e3}), and by
expanding the determinant along the $i$-th row, $i=1, 2, \ldots,
\phi(d_1)$ and $n$, we obtain the equivalent expression
\begin{equation}\label{e4}
\det(\mathbf{J}+\mathbf{Q})=m\cdot \det \left[
\begin{array}{ll|ll|ll|l}
 m_2& .&  1& 1& .& .& . \\[0.1cm]
.& m_2&  1& 1& .& .& . \\[0.1cm]
\hline
1& 1&  m_3& .& 1& 1& . \\[0.1cm]
1& 1&  .& m_3& 1& 1& . \\[0.1cm]
\hline
.& .&  1 & 1& m_4& .& 1 \\[0.1cm]
.& .&  1 & 1& .& m_4& 1 \\[0.1cm]
\hline
.& .&  . & .& 1& 1& m_5 \\[0.1cm]
\end{array}\right ],\end{equation}
where $m={m_1}^{\phi(d_1)}\cdot {m_6}^{\phi(d_6)}$.

In order to compute the new determinant on the right-hand side,
denoted by $D$, we apply the following row and column operations:
We subtract column $j$ from column $j+1$, $j=1, 3, 5$, and
subsequently we add row $i+1$ to row $i$, $i=1, 3, 5$. It is not
difficult to see that, stage by stage, the rows and columns are
``emptied" until finally the following determinant
$$D=\prod\limits_{i=2}^{5}{m_i}^{\phi(d_i)}\cdot \det \left[
\begin{array}{cc|cc|cc|c}
 1& .&  \lambda_2^{-1} & .& .& .& . \\[0.1cm]
.& 1&  \frac{1}{m_2} & .& .& .& . \\[0.1cm]
\hline
 \lambda_3^{-1} & . &  1& . & \lambda_3^{-1} & . & . \\[0.1cm]
\frac{1}{m_3} & .&  . & 1& \frac{1}{m_3}& .& . \\[0.1cm]
\hline
.& .&  \lambda_4^{-1} & . & 1 & .& \lambda_4^{-1} \\[0.1cm]
.& .&   \frac{1}{m_4} & .& . & 1 & \frac{1}{m_4} \\[0.1cm]
\hline
.& .&  . & .&  \lambda_5^{-1} & . & 1 \\[0.1cm]
\end{array}\right ]\\[0.3cm]$$
is obtained. Expanding along the columns 2, 4 and 6, it follows
that
$$D=\prod\limits_{i=2}^{5}{m_i}^{\phi(d_i)}\cdot \det \left[
\begin{array}{cccc}
 1& \lambda_2^{-1}& .& . \\[0.1cm]
\lambda_3^{-1} & 1 & \lambda_3^{-1} & .\\[0.1cm]
.& \lambda_4^{-1} & 1 & \lambda_4^{-1} \\[0.1cm]
.&  .& \lambda_5^{-1} & 1 \\[0.1cm]
\end{array}\right ]$$
Taking out the common factors  $\lambda_i^{-1}$ of  $i$-th row,
$i=1, 2, 3, 4$, we obtain
$$\begin{array}{lll} D&=&\Phi^{-1}\prod\limits_{i=2}^{5}{m_i}^{\phi(d_i)}\cdot \det \left[
\begin{array}{cccc}
 \lambda_2 & 1& .& . \\[0.1cm]
1& \lambda_3 & 1 & . \\[0.1cm]
.& 1 & \lambda_4 & 1 \\[0.1cm]
.&  .& 1 & \lambda_5 \\[0.1cm]
\end{array}\right ]\\[1.5cm] &=&\Phi^{-1} \prod\limits_{i=2}^{5}{m_i}^{\phi(d_i)}\cdot
\Big(\Phi+\sum\limits_{\Lambda}
\det\mathbf{A}(\Lambda)\lambda_2^{t_2}
\lambda_3^{t_3}\lambda_4^{t_4} \lambda_5^{t_{5}}\Big),
\end{array}$$
where $t_i\in \{0, 1\}$, $i=2, 3, 4, 5$, and the summation is
over all induced subgraphs $\Lambda$ of
$\overline{D}(12)\setminus \{d_1, d_6\}$ to the vertices
$\{d_{i_1}, \ldots, d_{i_s}\}$ corresponding to
$t_{i_1}=\cdots=t_{i_s}=0$. If this is substituted in Eq.
(\ref{e4}) and the products are put together, then we obtain
$$\det(\mathbf{J}+\mathbf{Q})=\Phi^{-1}\prod_{i=1}^{6}m_i^{\phi(d_i)}\cdot
\Big(\Phi+\sum\limits_{\Lambda}
\det\mathbf{A}(\Lambda)\lambda_2^{t_2}
\lambda_3^{t_3}\lambda_4^{t_4} \lambda_5^{t_{5}}\Big).$$ The
final assertion immediately follows from Theorem \ref{th-1}. This
works in general, as we now demonstrate.

As we mentioned already the matrix $\mathbf{J}+\mathbf{Q}$ has
the form as Eq. (\ref{e2}), and by developing the determinant
several times along the rows $1, 2, \ldots, \phi(d_1)$ and $n$,
one gets
\begin{equation}\label{e5}
\det(\mathbf{J}+\mathbf{Q})=m\cdot \det \left[
\begin{array}{cccc}
D_{22} & D_{23} & \ldots & D_{2, k-1} \\[0.1cm]
D_{32} & D_{33} & \ldots & D_{3, k-1} \\[0.1cm]
\vdots & \vdots & \ddots & \vdots \\[0.1cm]
D_{k-1, 2} & D_{k-1, 3} & \ldots & D_{k-1, k-1} \\[0.1cm]
\end{array}\right ],\end{equation}
where $m={m_1}^{\phi(d_1)}\cdot {m_k}^{\phi(d_k)}$.

In what follows, $D$ denotes the new determinant on the
right-hand side of Eq. (\ref{e5}). In order to compute this
determinant, we apply the following row and column operations: We
subtract column $j$ from column $j+r$:
$$\left\{\begin{array}{ll} j=1+\sum\limits_{l=2}^{h}\phi(d_l), \ h=1,2, \ldots, k-2,\\[0.3cm]
r=1, 2, \ldots, \phi(d_{h+1})-1,\end{array} \right.$$ and
subsequently we add row $i+s$ to row $i$:
$$\left\{\begin{array}{ll} i=1+\sum\limits_{l=2}^{h}\phi(d_l), \ h=1,2, \ldots, k-2,\\[0.3cm]
s=1, 2, \ldots, \phi(d_{h+1})-1.\end{array} \right.$$ (Note that,
when $m>n$, we assume that $\sum_{i=m}^na_i=0$). Using the above
operations, it is easy to see that
$$D=\det \left[
\begin{array}{cccc}
M_{22} & M_{23} & \ldots & M_{2, k-1} \\[0.1cm]
M_{32} & M_{33} & \ldots & M_{3, k-1} \\[0.1cm]
\vdots & \vdots & \ddots & \vdots \\[0.1cm]
M_{k-1, 2} & M_{k-1, 3} & \ldots & M_{k-1, k-1} \\[0.1cm]
\end{array}\right ],$$
where $M_{ij}$ is a matrix of size $\phi(d_i)\times \phi(d_j)$
with
$$ M_{ij}=\left\{\begin{array}{ll} m_i\mathbf{I} &  \mbox{if} \  \  i=j,\\[0.2cm]
0 &  \mbox{if} \ \  i\neq j, \ \ d_i|d_j \  \mbox{or} \ d_j|d_i,\\[0.2cm]
\phi(d_i)\mathbf{E}_{1, 1}+\mathbf{E}_{2,
1}+\cdots+\mathbf{E}_{\phi(d_i), 1} & \mbox{otherwise},
\end{array} \right.$$
where $\mathbf{I}$ is the identity matrix and $\mathbf{E}_{i,j}$
denotes the square matrix having $1$ in the $(i,j)$ position and
$0$ elsewhere.

Therefore, taking out the common factors and developing the
determinant along the columns $j$, with
$$j \neq 1+\sum\limits_{l=2}^{h}\phi(d_l), \ \ \  h=1,2, \ldots, k-2,$$ one
gets
\begin{equation}\label{e6} D=\Phi^{-1}\prod_{i=2}^{k-1}{m_i}^{\phi(d_i)}\cdot\det
\left[
\begin{array}{cccc}
a_{22} & a_{23} & \ldots & a_{2, k-1} \\[0.1cm]
a_{32} & a_{33} & \ldots & a_{3, k-1} \\[0.1cm]
\vdots & \vdots & \ddots & \vdots \\[0.1cm]
a_{k-1, 2} & a_{k-1, 3} & \ldots & a_{k-1, k-1} \\[0.1cm]
\end{array}\right ],\end{equation}
where $$ a_{ij}=\left\{\begin{array}{cl}\lambda_i &  \mbox{if} \  \  i=j,\\[0.2cm]
0 &  \mbox{if} \ \  i\neq j, \ \ d_i|d_j \  \mbox{or} \ d_j|d_i,\\[0.2cm]
1 & \mbox{otherwise}.
\end{array} \right.$$
As the reader might have noticed, the following matrix
$$\left[
\begin{array}{cccc}
0 & a_{23} & \ldots & a_{2, k-1} \\[0.1cm]
a_{32} & 0 & \ldots & a_{3, k-1} \\[0.1cm]
\vdots & \vdots & \ddots & \vdots \\[0.1cm]
a_{k-1, 2} & a_{k-1, 3} & \ldots & 0 \\[0.1cm]
\end{array}\right],$$
is exactly the adjacency matrix of the graph
$\Gamma=\overline{D}(n)\setminus \{d_1, d_k\}$. Consequently, we
get
$$\det \left[
\begin{array}{cccc}
\lambda_2 & a_{23} & \ldots & a_{2, k-1} \\[0.1cm]
a_{32} & \lambda_3 & \ldots & a_{3, k-1} \\[0.1cm]
\vdots & \vdots & \ddots & \vdots \\[0.1cm]
a_{k-1, 2} & a_{k-1, 3} & \ldots & \lambda_{k-1} \\[0.1cm]
\end{array}\right ]=\Phi+\sum_{\Lambda}
\det\mathbf{A}(\Lambda)\lambda_2^{t_2}\lambda_3^{t_3}\cdots
\lambda_{k-1}^{t_{k-1}},$$ where $t_i\in \{0, 1\}$, $i=2, 3,
\ldots, k-1$, and the summation is over all induced subgraphs
$\Lambda$ of $\Gamma$ to the vertices $\{d_{i_1}, \ldots,
d_{i_s}\}$ corresponding to $t_{i_1}=\cdots=t_{i_s}=0$. This is
substituted in Eq. (\ref{e6}):
$$D=\Phi^{-1}\prod_{i=2}^{k-1}{m_i}^{\phi(d_i)}\cdot \Big(\Phi+\sum_{\Lambda}
\det\mathbf{A}(\Lambda)\lambda_2^{t_2}\lambda_3^{t_3}\cdots
\lambda_{k-1}^{t_{k-1}}\Big).$$ Again, if this is substituted in
Eq. (\ref{e5}) and the sums are put together, then we obtain
$$\det(\mathbf{J}+\mathbf{Q})=\Phi^{-1}\prod_{i=1}^{k}{m_i}^{\phi(d_i)}\cdot \Big(\Phi+\sum_{\Lambda}
\det\mathbf{A}(\Lambda)\lambda_2^{t_2}\lambda_3^{t_3}\cdots
\lambda_{k-1}^{t_{k-1}}\Big).$$ The required result now follows
immediately from Theorem \ref{th-1}.
\end{proof}

A computer check has confirmed that our expression is correct for
all values of $n$ up to 100 and so we are confident in the
correctness of our results.
\begin{corollary}\label{corollary-202}
Let $n>2$ be an integer. Then $\kappa(\mathbb{Z}_n)$ is divisible
by $n$.
\end{corollary}
\begin{proof}
Using the above notation and Theorems \ref{result2}, we can see
that $\kappa(\mathbb{Z}_n)$ is divisible by
$$\frac{{m_1}^{\phi(d_1)}{m_k}^{\phi(d_k)}}{n^2}=\frac{n^{\phi(n)}n^{\phi(1)}}{n^2}=n^{\phi(n)-1},$$
and since $\phi(n)\geqslant 2$ for all $n>2$, the result is
proved.
\end{proof}

In what follows, we assume that
$$\gamma_i:=\frac{n_i}{\phi(d_i)}, \ i=2, \ldots, k-1, \ \
\mbox{and} \ \ \Psi:=\prod_{i=2}^{k-1}\gamma_i.$$

We will denote by $\mathcal{P}(G^\#)$ the graph obtained by
deleting the vertex 1 (the identity element of $G$) from
$\mathcal{P}(G)$. For convenience, we will denote
$\kappa(\mathcal{P}(G^\#))$ as $\kappa(G^\#)$. In the same manner
as in the proof of Theorem \ref{result2}, we can prove the
following theorem.
\begin{theorem}\label{result-22} Let $n$ be a positive integer and $d_1>d_2>\cdots>d_k$ all divisors of
$n$. With the above notation, there holds
\begin{equation}\label{eq-new}
\kappa(\mathbb{Z}_n^\#)=\prod_{i=1}^{k-1}n_i^{\phi(d_i)}\Big(\Psi+
\sum_{\Lambda} \det\mathbf{A}(\Lambda)\gamma_2^{t_2}
\gamma_3^{t_3}\cdots \gamma_{k-1}^{t_{k-1}}\Big)/(\Psi
(n-1)^2),\end{equation} where $t_i\in \{0, 1\}$, $i=2, 3, \ldots,
k-1$, and the summation is over all induced subgraphs $\Lambda$ of
$\overline{D}(n)\setminus \{d_1, d_k\}$ to the vertices
$\{d_{i_1}, \ldots, d_{i_s}\}$ corresponding to
$t_{i_1}=\cdots=t_{i_s}=0$.
\end{theorem}

{\em Some Examples.} (1) Suppose that $n=p^f$, where $p$ is a
prime number and $f\in \mathbb{N}$. According to Lemma
\ref{complete-CSS}, the power graph $\mathcal{P}(\mathbb{Z}_n)$
is a complete graph and by Cayley result we obtain
$$\kappa(\mathbb{Z}_{n})=n^{n-2}.$$ Alternatively, we may use
Theorem \ref{result2}. With the notation as in Theorem
\ref{result2}, we have

$\bullet$  \  $k=f+1$ and $d_i=p^{f-i+1}$,  $i=1, 2, \ldots, f+1$;

$\bullet$  \  for every $i$, $m_i=n$;

$\bullet$  \ The divisor graph $D(n)$ is a complete graph, and so
$\overline{D}(n)$ is a null graph.

\noindent Substituting the above expressions into Eq.
(\ref{trees-cyclic}), we obtain
$$\begin{array}{lll} \kappa(\mathbb{Z}_{n}) & = & \prod\limits_{i=1}^{f+1}n^{\phi(p^{f-i+1})}\Big(\Phi+
\sum\limits_{\Lambda} \det\mathbf{A}(\Lambda)\lambda_2^{t_2}
\lambda_3^{t_3}\cdots \lambda_{f}^{t_{f}}\Big)/(\Phi n^2)\\[0.3cm]
& = & n \prod\limits_{i=1}^{f}n^{\phi(p^{f-i+1})}\big(\Phi+
0\big)/(\Phi n^2)\\[0.3cm]
& = & \frac{1}{n}\prod\limits_{i=1}^{f}n^{p^{f-i}(p-1)}\\[0.3cm] & = &
\frac{1}{n}n^{(p-1)\sum\limits_{i=1}^{f}p^{f-i}}
=\frac{1}{n}n^{p^f-1}=n^{n-2},

\end{array}$$
as required.

$(2)$ Suppose that $n=pq$, where $p<q$ are primes. Once again,
with the notation as in Theorem \ref{result2}, we have

$\bullet$  \  $d_1=n$, $d_2=q$, $d_3=p$ and $d_4=1$;

$\bullet$  \  By Lemma \ref{pre-1}, we obtain $$m_1=n, \ \
m_2=n-p+1, \ \ m_3=n-q+1 \ \ \mbox{and} \ \ m_4=n;$$

$\bullet$  \ The divisor graph $D(n)\setminus\{d_1, d_4\}$ is a
null graph, and so $\overline{D}(n)\setminus \{d_1, d_4\}$ is a
complete graph.

\noindent Substituting the above expressions into Eq.
(\ref{trees-cyclic}), we obtain
$$\begin{array}{lll} \kappa(\mathbb{Z}_{n}) & = & \prod\limits_{i=1}^{4}{m_i}^{\phi(d_i)}\Big(\Phi+
\sum\limits_{\Lambda} \det\mathbf{A}(\Lambda)\lambda_2^{t_2}
\lambda_3^{t_3}\Big)/(\Phi n^2)\\[0.3cm]
& = & \prod\limits_{i=1}^{3}{m_i}^{\phi(d_i)}\big(\Phi-1\big)/(\Phi n)\\[0.5cm]
& = &
n^{(p-1)(q-1)}(n-p+1)^{q-2}(n-q+1)^{p-2}(n-p-q+2).\end{array}$$

Similarly, substituting the above expressions into Eq.
(\ref{eq-new}), we obtain
$$\begin{array}{lll} \kappa(\mathbb{Z}_{n}^{\#}) & = & \prod\limits_{i=1}^{3}{n_i}^{\phi(d_i)}\Big(\Psi+
\sum\limits_{\Lambda} \det\mathbf{A}(\Lambda)\lambda_2^{t_2}
\lambda_3^{t_3}\Big)/(\Psi (n-1)^2)\\[0.3cm]
& = & \prod\limits_{i=1}^{3}{m_i}^{\phi(d_i)}\big(\Psi-1\big)/(\Psi (n-1)^2)\\[0.5cm]
& = &
(n-1)^{(p-1)(q-1)-1}(n-p)^{q-2}(n-q)^{p-2}(n-p-q+1).\end{array}$$
In particular, we have
$$\kappa(\mathbb{Z}_{2p})=\frac{1}{2}(2p)^p(2p-1)^{p-2} \ \
\mbox{and} \ \
\kappa(\mathbb{Z}_{2p}^\#)=\frac{1}{2}(2p-1)^{p-2}(2p-2)^{p-1},$$
where $p$ is an odd prime. For instance, if $n=6$, then easy
computations show that $\kappa(\mathbb{Z}_{6})=540$ and
$\kappa(\mathbb{Z}_{6}^\#)=40$. The power graph
$\mathcal{P}(\mathbb{Z}_{6})$ is depicted in Fig. 2.

\setlength{\unitlength}{3mm}
\begin{picture}(0,0)(-8,12)
\linethickness{0.7pt} %
\put(0.04,10){\circle*{0.3}}%
\put(3,10){\circle*{0.3}}%
\put(-1.2,7.5){\circle*{0.3}}%
\put(4.2,7.5){\circle*{0.3}}%
\put(1.5,5.5){\circle*{0.3}}%
\put(1.5,2.1){\circle*{0.3}}%
\put(0,10){\line(1,0){3}}%
\put(-1.2,7.5){\line(1,2){1.2}}%
\put(-1.2,7.5){\line(4,-3){2.7}}%
\put(3,10){\line(1,-2){1.2}}%
\put(1.5,5.5){\line(4,3){2.7}}%
\put(1.5,5.5){\line(1,3){1.5}}%
\put(1.5,5.5){\line(-1,3){1.5}}%
\put(-1.2,7.5){\line(1,0){5.3}}%
\put(-1.2,7.5){\line(5,3){4.3}}%
\put(4.2,7.5){\line(-1,-2){2.7}}%
\put(4.2,7.5){\line(-5,3){4.2}}%
\put(-1.2,7.5){\line(1,-2){2.7}}%
\put(1.5,2.1){\line(0,1){3.3}}%
\put(-0.3,10.5){\small$x^2$}%
\put(3,10.5){\small$x^4$}%
\put(-2.3,7.3){\small$x$}%
\put(4.8,7.3){\small$x^5$}%
\put(0.8,4.8){\small$1$}%
\put(1,0.7){\small$x^3$}%
\put(15.04,10){\circle*{0.3}}%
\put(18,10){\circle*{0.3}}%
\put(13.8,7.5){\circle*{0.3}}%
\put(19.2,7.5){\circle*{0.3}}%
\put(16.5,5.5){\circle*{0.3}}%
\put(16.5,2.1){\circle*{0.3}}%
\put(14.8,2.1){\circle*{0.3}}%
\put(13.1,2.1){\circle*{0.3}}%
\put(18.2,2.1){\circle*{0.3}}%
\put(19.9,2.1){\circle*{0.3}}%
\put(16,10){$\ldots$}%
\put(18.4,2){$\ldots$}%
\put(13.8,7.5){\line(1,2){1.2}}%
\put(13.8,7.5){\line(4,-3){2.7}}%
\put(18,10){\line(1,-2){1.2}}%
\put(16.5,5.5){\line(4,3){2.7}}%
\put(16.5,5.5){\line(1,3){1.5}}%
\put(16.5,5.5){\line(-1,3){1.5}}%
\put(13.8,7.5){\line(1,0){5.3}}%
\put(13.8,7.5){\line(5,3){4.3}}%
\put(19.2,7.5){\line(-5,3){4.2}}%
\put(16.5,2.1){\line(0,1){3.4}}%
\put(16.5,5.5){\line(-1,-2){1.75}}%
\put(16.5,5.5){\line(1,-2){1.75}}%
\put(16.5,5.5){\line(-1,-1){3.4}}%
\put(16.5,5.5){\line(1,-1){3.4}}%
\put(14.7,10.5){\small $x^2$}%
\put(18,10.5){\small $x^{n-2}$}%
\put(12.7,7.3){\small $x$}%
\put(19.8,7.3){\small $x^{n-1}$}%
\put(17.3,5.2){\small $1$}%
\put(13.8,0.8){\small $xy$}%
\put(12.3,0.8){\small $y$}%
\put(15.6,0.8){\small $x^2y$}%
\put(17.6,0.8){\small $x^3y$}%
\put(19.8,0.8){\small $x^{n-1}y$}%
\put(11,-1){{\small {\bf Fig. 3.} The power graph
$\mathcal{P}(D_{2n})$.}} \put(-6,-1){{\small {\bf Fig. 2.} The
power graph $\mathcal{P}(\mathbb{Z}_{6})$.}}
\end{picture}\\[4cm]

(3) Let $D_{2n}$ be the dihedral group of order $2n$, which is
defined by
$$D_{2n}=\langle x, y \ | \ x^{n}=y^2=1, x^y=x^{-1}\rangle.$$
The power graph $\mathcal{P}(D_{2n})$ is depicted in Fig. 3. It
is easy to see that, the edges $\{1, x^iy\},  \ \ i=0, 1, \ldots,
n-1,$ are cut-edges and by Theorem \ref{th-2}, we obtain
$\kappa(D_{2n})=\kappa(\mathbb{Z}_n)$.

We conclude this section with the following corollary, which
derived from before observations.
\begin{corollary} Suppose that $|G|=pq$, where $p$ and $q$ are
primes (not necessarily distinct).  Then one of the following
possibilities must occurs:
\begin{itemize}
\item[{$(a)$}] $p=q$, $G\cong \mathbb{Z}_{p^2}$, and
$\kappa(G)=p^{2(p^2-2)}$.
\item[{$(b)$}] $p=q$, $G\cong \mathbb{Z}_{p}\times
\mathbb{Z}_{p}$, and $\kappa(G)=p^{(p+1)(p-2)}$.
\item[{$(c)$}] $p\neq q$, $G\cong \mathbb{Z}_{p}\times
\mathbb{Z}_{q}\cong \mathbb{Z}_{pq}$, and
$$\kappa(G)=n^{(p-1)(q-1)}(n-p+1)^{q-2}(n-q+1)^{p-2}(n-p-q+2),$$
where $n=pq$.
\item[{$(d)$}] $q<p$, $p\equiv 1\pmod{q}$, $G\cong \mathbb{Z}_{p}\rtimes
\mathbb{Z}_{q}$, and $\kappa(G)=p^{p-2}q^{p(q-2)}$.
\end{itemize}
\end{corollary}
\section{The Tree-number of $\mathcal{P}(Q_{4n})$}
In this section, we focus our attention on the generalized
quaternion groups. Let $Q_{4n}$ denote the generalized quaternion
group of order $4n$, which can be presented by
$$Q_{4n}=\langle x, y \ | \ x^{2n}=1, y^2=x^n, x^y=x^{-1}\rangle.$$
First of all, we recall that every element in $Q_{4n}$ can be
written uniquely as $x^iy^j$ where $0\leqslant i\leqslant 2n-1$
and $0\leqslant j\leqslant 1$, and so the order of $Q_{4n}$ is
exactly $4n$. Moreover, for a natural number $n$, the graph
$\mathcal{P}(Q_{4n}^\#)$ consists of the power graph
$\mathcal{P}(\mathbb{Z}_{2n}^\#)$ and $n$ triangles sharing a
common vertex (i.e., $x^n$). Indeed $x^n$ is a {\em cut vertex}
in $\mathcal{P}(Q_{4n}^\#)$.

It is well known that the group $Q_{4n}$ has a unique minimal
subgroup if and only if $n$ is a power of $2$. Furthermore, in
the case when $n$ is a power of $2$, the graph
$\mathcal{P}(Q_{4n}^\#)$ has the following form:
$$\mathcal{P}(Q_{4n}^\#)=K_1\vee (K_{2n-2}\oplus
\underbrace{K_2\oplus K_2\oplus \cdots \oplus K_2}_{n {\rm
-times}}).$$ As a matter of fact, it consists of a complete graph
on $2n-1$ vertices and $n$ triangles sharing a common vertex (the
unique involution, i.e., $x^n$). The graph
$\mathcal{P}(Q_{4n}^\#)$, for the case $n$ is a power of 2, is
depicted in Fig. 4. Note that, for every element $g\in
Q_{4n}\setminus \{1\}$, the subgroup $\langle g \rangle$ contains
the unique involution $x^n$, and so $x^n\sim g$ in
$\mathcal{P}(Q_{4n}^\#)$, hence in $\mathcal{P}(Q_{4n}^\#)$ the
unique involution $x^n$ is the {\em only} vertex of degree $4n-2$.

\begin{theorem}\label{result-333}
If $n$ is a natural number, then the tree-number of the power
graph $\mathcal{P}(Q_{4n}^\#)$ is given by the formula
$$\kappa(Q_{4n}^\#)=3^n\cdot \kappa(\mathbb{Z}_{2n}^\#).$$
\end{theorem}
\begin{proof} Let
$\Gamma=\mathcal{P}(Q_{4n}^\#)$. As we mentioned already, $\Gamma$
is a connected graph and contains $x^n$ as a cut vertex. Now, we
consider the following set of vertices. Let
$$\begin{array}{l} S_0=\{x, x^2, \ldots, x^{2n-1}\},\\[0.2cm]
S_i=\{yx^{i-1}, yx^{n+i-1}, x^n\}, \ \  i=1, 2,  \ldots, n.
\end{array} $$
Then
$$\Gamma=\Gamma[S_0]+\Gamma[S_1]+\cdots+\Gamma[S_n],$$
in which $\Gamma[S_i]\cap\Gamma[S_j]=(\{x^n\}, \emptyset)$,
$0\leqslant i<j\leqslant n$. Notice that, we have
$\Gamma[S_0]=\mathcal{P}(\mathbb{Z}_{2n}^\#)$. Finally, from
Theorem \ref{th-222}, it follows that
$$\begin{array}{lll}\kappa(\Gamma)=\prod\limits_{i=0}^{n}\kappa(\Gamma[S_i])& = &
\kappa(\Gamma[S_0])\prod\limits_{i=1}^{n}\kappa(\Gamma[S_i])
\\[0.5cm] &=&\kappa(\mathbb{Z}_{2n}^\#)\prod\limits_{i=1}^{n}\kappa(K_3)=\kappa(\mathbb{Z}_{2n}^\#)\cdot
3^n,\\ \end{array}$$ as claimed.
\end{proof}

\vspace{2cm} \setlength{\unitlength}{3mm}
\begin{picture}(0,0)(-1,5)
\linethickness{0.7pt} %
\put(15.04,10){\circle*{0.3}}%
\put(18,10){\circle*{0.3}}%
\put(13.8,7.5){\circle*{0.3}}%
\put(19.2,7.5){\circle*{0.3}}%
\put(16.5,5.5){\circle*{0.3}}%
\put(16,10){$\ldots$}%
\put(13.8,7.5){\line(1,2){1.2}}%
\put(13.8,7.5){\line(4,-3){2.7}}%
\put(18,10){\line(1,-2){1.2}}%
\put(16.5,5.5){\line(4,3){2.7}}%
\put(16.5,5.5){\line(1,3){1.5}}%
\put(16.5,5.5){\line(-1,3){1.5}}%
\put(13.8,7.5){\line(1,0){5.3}}%
\put(13.8,7.5){\line(5,3){4.3}}%
\put(19.2,7.5){\line(-5,3){4.2}}%
\put(16.5,5.5){\line(-1,-5){0.65}}%
\put(16.5,5.5){\line(1,-5){0.65}}%
\put(17.2,2.1){\circle*{0.3}}%
\put(15.8,2.1){\circle*{0.3}}%
\put(15.8,2.1){\line(1,0){1.5}}%
\put(17.6,2.1){$\ldots$}%
\put(14.3,2.1){$\ldots$}%
\put(16.5,5.5){\line(-4,-3){4.5}}%
\put(16.5,5.5){\line(-4,-5){2.7}}%
\put(12,2.1){\circle*{0.3}}%
\put(13.8,2.1){\circle*{0.3}}%
\put(12,2.1){\line(1,0){1.8}}%
\put(16.5,5.5){\line(4,-3){4.5}}%
\put(16.5,5.5){\line(4,-5){2.7}}%
\put(21,2.1){\circle*{0.3}}%
\put(19.2,2.1){\circle*{0.3}}%
\put(19.2,2.1){\line(1,0){1.8}}%
\put(13.7,10.5){\small $x^{n-2}$}%
\put(18,10.5){\small $x^{n+2}$}%
\put(11.5,7.3){\small $x^{n-1}$}%
\put(19.8,7.3){\small $x^{n+1}$}%
\put(14.7,5.2){\small $x^n$}%
\put(11.3,1.2){\tiny $y$}%
\put(12.8,1.2){\tiny $yx^n$}%
\put(14.5,1.2){\tiny $yx^i$}%
\put(16,1.2){\tiny $yx^{n+i}$}%
\put(18.4,1.2){\tiny $yx^{n-1}$}%
\put(20.9,1.2){\tiny $yx^{2n-1}$}%
\put(10.5,-1){{\small {\bf Fig. 4.} The graph
$\mathcal{P}(Q_{4n}^\#)$.}}
\end{picture}\\[2cm]

As an immediate consequence of Theorem \ref{result-333}, Lemma
\ref{complete-CSS} and Cayley's Theorem, we have the following
corollary.
\begin{corollary}\label{corollary-1}
If $n$ is a power of $2$, then the tree-number of the power graph
$\mathcal{P}(Q_{4n}^\#)$ is given by the formula
$$\kappa(Q_{4n}^\#)=3^n\cdot (2n-1)^{2n-3}.$$
\end{corollary}

Now, we return to the power graph $\mathcal{P}(Q_{4n})$.
\begin{theorem}\label{result-222}
If $n$ is a power of $2$, then the tree-number of the power graph
$\mathcal{P}(Q_{4n})$ is given by the formula
$$\kappa(Q_{4n})=2^{5n-1}\cdot n^{2n-2}.$$
\end{theorem}
\begin{proof} We follow an
algebraic approach, again. In the sequel, for the sake of
convenience we will consider the following labeling of the
vertices of $\mathcal{P}(Q_{4n})$:
$$\begin{array}{lll} V & = & \{x, x^2, \ldots, x^{n-1}, x^{n+1}, \ldots,
x^{2n-1}, y, yx^n, yx, yx^{n+1},\\[0.3cm] & & \ \ \  \ldots,  yx^i, yx^{n+i},
\ldots, yx^{n-1}, yx^{2n-1}, x^n, 1\}.\end{array}$$

Now, the Laplacian matrix $\mathbf{Q}$ is given by:
$$\mathbf{Q}=\left[ \begin{array}{cccc|cc|c|cc|c|c}
\sigma & -1 & \cdots & -1& . & . & \cdots & . & . & -1  & -1\\[0.1cm]
-1 & \sigma & \cdots & -1& . & .&   \cdots & .& . & -1 & -1\\[0.1cm]
\vdots & \vdots & \ddots & -1& \vdots & \vdots  & \cdots & \vdots &  \vdots & \vdots & \vdots \\[0.1cm]
-1 & -1 & \cdots& \sigma & .& .&   \cdots & .& .& -1 &-1\\[0.1cm]
\hline
. & . & \cdots & .& 3& -1&  \cdots & .& .& -1 & -1\\[0.1cm]
. & . & \cdots & .& -1& 3& \cdots & .& .& -1 & -1 \\[0.1cm]
\hline
\vdots & \vdots & \cdots& \vdots& \vdots& \vdots& \ddots& \vdots& \vdots& \vdots  & \vdots \\[0.1cm]
\hline
. & . & \cdots & .& .& .&  \cdots & 3& -1& -1 & -1\\[0.1cm]
. & . & \cdots & .& .& .&  \cdots & -1& 3& -1 & -1 \\[0.1cm]
\hline
-1 & -1 & \cdots & -1& -1& -1&   \cdots& -1 & -1& \tau & -1\\[0.1cm]
\hline
-1 & -1 & \cdots & -1& -1& -1&  \cdots& -1  & -1& -1& \tau \\[0.1cm]
\end{array}\right ],$$
where $\sigma=2n-1$ and $\tau=4n-1$, and hence we obtain
$$\mathbf{J}+\mathbf{Q}=\left[ \begin{array}{cccc|cc|c|cc|c|c}
2n & . & \cdots & .& 1 & 1 &  \cdots & 1 & 1 & .& .\\[0.1cm]
. & 2n & \cdots & . & 1 & 1&   \cdots & 1& 1 & . & .\\[0.1cm]
\vdots & \vdots & \ddots & .& \vdots & \vdots & \cdots & \vdots &  \vdots & \vdots & \vdots\\[0.1cm]
. & . & \cdots& 2n & 1& 1&  \cdots & 1& 1& . & .\\[0.1cm]
\hline
1 & 1 & \cdots & 1& 4& .&  \cdots & 1& 1& . & .\\[0.1cm]
1 & 1 & \cdots & 1& .& 4&  \cdots & 1& 1& . & .\\[0.1cm]
\hline
\vdots & \vdots & \cdots& \vdots& \vdots& \vdots&  \ddots& \vdots& \vdots& \vdots  & \vdots\\[0.1cm]
\hline
1 & 1 & \cdots & 1& 1& 1&  1& 4& .& . & .\\[0.1cm]
1 & 1 & \cdots & 1& 1& 1& 1& .& 4& .& . \\[0.1cm]
\hline
. & . & \cdots & .& .& .&  \cdots & .& .& 4n & .\\[0.1cm]
\hline
. & . & \cdots & .& .& .&  \cdots & .& .&. & 4n \\[0.1cm]
\end{array}\right ].$$

To compute $\det(\mathbf{J}+\mathbf{Q})$, expand
$\det(\mathbf{J}+\mathbf{Q})$ with respect to its two last rows:
\begin{equation}\label{e1} \det(\mathbf{J}+\mathbf{Q})=(4n)^2\cdot \det
\left[\begin{array}{cccc|cc|c|cc}
2n & . & \cdots & .& 1 & 1 &   \cdots & 1 & 1 \\[0.1cm]
. & 2n & \cdots & . & 1 & 1& \cdots & 1& 1 \\[0.1cm]
\vdots & \vdots & \ddots & .& \vdots & \vdots &  \cdots & \vdots &  \vdots  \\[0.1cm]
. & . & \cdots& 2n & 1& 1&  \cdots & 1& 1\\[0.1cm]
\hline
1 & 1 & \cdots & 1& 4& .&  \cdots & 1& 1\\[0.1cm]
1 & 1 & \cdots & 1& .& 4&  \cdots & 1& 1 \\[0.1cm]
\hline
\vdots & \vdots & \cdots& \vdots& \vdots& \vdots&  \ddots& \vdots& \vdots \\[0.1cm]
\hline
1 & 1 & \cdots & 1& 1& 1& .& 4& . \\[0.1cm]
1 & 1 & \cdots & 1& 1& 1&  . & .& 4\\[0.1cm]
\end{array}\right].\end{equation}

Now, we apply the following row and column operations in the new
determinant, which is denoted by $D$:

We subtract row 1 from row $i$, $i=2, 3, \ldots, 2n-2$, and
subsequently we subtract column 1 from column $j$, $j=2, 3,
\ldots, 2n-2$. It is not too difficult to see that, step by step,
the rows and columns are ``emptied" until finally the determinant
$$D=\det \left[\begin{array}{cccc|cc|c|cc}
2n & -2n & \cdots & -2n& 1 & 1 &   \cdots & 1 & 1 \\[0.1cm]
-2n & 4n & \cdots & 2n & . & .&  \cdots & .& . \\[0.1cm]
\vdots & \vdots & \ddots & \vdots & \vdots & \vdots  & \cdots & \vdots &  \vdots  \\[0.1cm]
-2n & 2n & \cdots& 4n & .& .&   \cdots & .& .\\[0.1cm]
\hline
1 & . & \cdots & . & 4& .&   \cdots & 1& 1\\[0.1cm]
1 & . & \cdots & . & .& 4&  \cdots & 1& 1 \\[0.1cm]
\hline
\vdots & \vdots & \cdots& \vdots& \vdots& \vdots&  \ddots& \vdots& \vdots \\[0.1cm]
\hline
1 & . & \cdots & .& 1& 1&   \cdots & 4& . \\[0.1cm]
1 & . & \cdots & .& 1& 1&   \cdots & .& 4\\[0.1cm]
\end{array}\right],$$
is obtained. Again  we subtract row 1 from row $i$, $i=2n-1, 2n,
\ldots, 4n-2$, and we obtain
$$D=\det \left[\begin{array}{cccc|cc|c|cc}
2n & -2n & \cdots & -2n& 1 & 1  & \cdots & 1 & 1 \\[0.1cm]
-2n & 4n & \cdots & 2n & . & . & \cdots & .& . \\[0.1cm]
\vdots & \vdots & \ddots & \vdots & \vdots & \vdots  & \cdots & \vdots &  \vdots  \\[0.1cm]
-2n & 2n & \cdots& 4n & .& .& \cdots & .& .\\[0.1cm]
\hline
1-2n & 2n & \cdots & 2n & 3& -1& \cdots & .& .\\[0.1cm]
1-2n & 2n & \cdots & 2n & -1& 3&   \cdots & .& . \\[0.1cm]
\hline
\vdots & \vdots & \cdots& \vdots& \vdots& \vdots&  \ddots& \vdots& \vdots \\[0.1cm]
\hline
1-2n & 2n & \cdots & 2n& .& . & .& 3& -1 \\[0.1cm]
1-2n & 2n & \cdots & 2n& .& .& .& -1& 3\\[0.1cm]
\end{array}\right].$$

In the following, $R_i$ and $C_j$ respectively designate the row
$i$ and the column $j$ of a matrix. Applying
$-\frac{1}{2}(R_{2n-1}+R_{2n}+\cdots+R_{4n-2})\longrightarrow
R_1$, to $D$,  we conclude that
$$
D=\det \left[\begin{array}{cccc|cc|c|cc}
\alpha & \beta & \cdots & \beta & . & . &   \cdots & . & . \\[0.1cm]
-2n & 4n & \cdots & 2n & . & .&   \cdots & .& . \\[0.1cm]
\vdots & \vdots & \ddots & \vdots& \vdots & \vdots &  \cdots & \vdots &  \vdots  \\[0.1cm]
-2n & 2n & \cdots& 4n & .& .&  \cdots & .& .\\[0.1cm]
\hline
1-2n & 2n & \cdots & 2n & 3& -1&   \cdots & .& .\\[0.1cm]
1-2n & 2n & \cdots & 2n & -1& 3&   \cdots & .& . \\[0.1cm]
\hline
\vdots & \vdots & \cdots& \vdots& \vdots& \vdots& \ddots& \vdots& \vdots \\[0.1cm]
\hline
1-2n & 2n & \cdots & 2n& .& .&  .& 3& -1 \\[0.1cm]
1-2n & 2n & \cdots & 2n& .& .&  .& -1& 3\\[0.1cm]
\end{array}\right],
$$
where $\alpha=2n^2+n$ and $\beta=-2n-2n^2$. Finally, by easy
computations, we can argue as follows:
$$
\begin{array}{lll}
D&=&\det \left[\begin{array}{cccc}
2n^2+n & -2n-2n^2 & \cdots & -2n-2n^2  \\[0.1cm]
-2n & 4n & \cdots & 2n \\[0.1cm]
\vdots & \vdots & \ddots & \vdots \\[0.1cm]
-2n & 2n & \cdots& 4n \\[0.1cm]
\end{array}\right]\cdot 2^{3n}\\[1.5cm]
&& \ \mbox{ ($C_1\longrightarrow C_j$, $j=2, 3, \ldots, 2n-2.$)
}\\
\end{array}$$
$$
\begin{array}{lll}
&=&\det \left[\begin{array}{cccc}
2n^2+n & -n & \cdots & -n  \\[0.1cm]
-2n & 2n & \cdots & . \\[0.1cm]
\vdots & \vdots & \ddots & \vdots \\[0.1cm]
-2n & . & \cdots& 2n \\[0.1cm]
\end{array}\right]\cdot 2^{3n}\\[1.5cm]
&&\mbox{($\frac{1}{2}(R_2+R_3+\cdots+R_{2n-2})\longrightarrow
R_1$)
}\\[0.5cm]
&=&\det \left[\begin{array}{cccc}
4n & . & \cdots & .  \\[0.1cm]
-2n & 2n & \cdots & . \\[0.1cm]
\vdots & \vdots & \ddots & \vdots \\[0.1cm]
-2n & . & \cdots& 2n \\[0.1cm]
\end{array}\right]\cdot 2^{3n}\\[1.5cm]
& = & (4n)(2n)^{2n-3}\cdot 2^{3n}=2^{5n-1}\cdot n^{2n-2}.\\
\end{array}$$
If this is substituted in Eq. (\ref{e1}), then we obtain
$$ \det(\mathbf{J}+\mathbf{Q})=(4n)^2\cdot 2^{5n-1}\cdot n^{2n-2}.$$
Now, from Theorem \ref{th-1}, it follows that
$$\kappa(Q_{4n})=\frac{\det(\mathbf{J}+\mathbf{Q})}{(4n)^2}= 2^{5n-1}\cdot n^{2n-2},$$
and the proof is complete. \end{proof}
\section{Some Applications}
\subsection{Groups With Small Tree-number} We begin with recalling
a definition in Graph Theory. A {\em star} is a tree consisting of
one vertex adjacent to all the others.
\begin{proposition}\label{prop-1} Let $G$ be a finite group. Then the following
statements are equivalent:
\begin{itemize}
\item[{\rm (a)}] $G$ is an elementary
abelian $2$-group.
\item[{\rm (b)}] the power graph $\mathcal{P}(G)$ is a star graph.
\item[{\rm (c)}] $\kappa(G)=1$.
\end{itemize}
\end{proposition}
\begin{proof} We may assume that $|G|\geqslant 3$, since otherwise all
the statements are trivially true.

${\rm (a)} \Rightarrow {\rm (b)}$ Suppose that $G$ is an
elementary abelian $2$-group.  If $x$ and $y$ are two distinct
non-trivial elements, then $x\notin \langle y\rangle$ and
$y\notin \langle x\rangle$, and so by the definition $x\nsim y$.
This shows that the set of involutions of $G$ is an independent
set in $\mathcal{P}(G)$. In other words, the power graph
$\mathcal{P}(G)$ is a star graph.

${\rm (b)} \Rightarrow {\rm (c)}$  It is a straightforward
verification.

${\rm (c)} \Rightarrow {\rm (a)}$ If $\kappa(G)=1$, then certainly
$\mathcal{P}(G)$ is a tree. We observe that if there is an
element $x$ of order $o(x)\geqslant 3$, then  $1\sim x\sim
x^2\sim 1$ is a cycle in $\mathcal{P}(G)$ and this contradicts
the fact that $\mathcal{P}(G)$ is a tree. Thus each element of
$G\setminus \{1\} $ has order 2, and hence $G$ is an elementary
abelian $2$-group, as required. \end{proof}

\begin{theorem}\label{th-small} Let $G$ be a nontrivial finite group. Then  $\kappa(G)<5^3$ if and only if
one of the following occurs:
\begin{itemize}
\item[{\rm (a)}] $\omega(G)=\{1, 2\}$ and $G$ is an elementary
abelian $2$-group.
\item[{\rm (b)}] $\omega(G)=\{1, 3\}$  and $G\cong \mathbb{Z}_3$ or $\mathbb{Z}_3\times
\mathbb{Z}_3$.

\item[{\rm (c)}] $\omega(G)=\{1, 2, 3\}$ and $G\cong \mathbb{S}_3\cong \mathbb{Z}_3\rtimes \mathbb{Z}_2$,
$(\mathbb{Z}_3\times \mathbb{Z}_3)\rtimes \mathbb{Z}_2$, or
$\mathbb{A}_4\cong(\mathbb{Z}_2\times \mathbb{Z}_2) \rtimes
\mathbb{Z}_3$

\item[{\rm (d)}] $\omega(G)=\{1, 2, 4\}$  and $G\cong \mathbb{Z}_4$  or
 $D_8$.
\end{itemize}
\end{theorem}
\begin{proof} We need only prove the
necessity. Let $G$ be a finite group such that $\kappa(G)<5^3$.
Then, from Theorem \ref{th-4}, we have
$$125>\kappa(G)\geqslant \prod_{p\in \pi(G)}p^{p-2},$$
which implies that $G$ is a $\{2, 3\}$-group with spectrum
$$\omega(G)\subseteq \{2^\alpha\cdot 3^\beta \ | \ \alpha\geqslant
0, \beta\geqslant 0\}.$$ Let $G_p\in {\rm Syl}_p(G)$, where $p\in
\{2, 3\}$. If $\mu(G_2)\geqslant 8$ (resp. $\mu(G_3)\geqslant 9$),
then by Theorem \ref{th-4}, $\kappa(G)\geqslant 8^6>125$ (resp.
$\kappa(G)\geqslant 9^7>125$), which is again a contradiction.
Hence, we conclude that $\omega(G)\subseteq \{1, 2, 3, 4, 6,
12\}$. If $G$ contains an element of order 6, say $x$, then
$\Omega=\{1, x, x^2, x^4, x^5\}$ is a clique in $\mathcal{P}(G)$,
and by Theorem \ref{th-5}, it follows that $\kappa(G)\geqslant
5^3$, which is a contradiction. This forces $\omega(G)\subseteq
\{1, 2, 3, 4\}$. We now consider five cases separately.

{\em Case $1$.} $\omega(G)=\{1, 2\}$. In this case, $G$ is an
elementary abelian $2$-group and by Proposition \ref{prop-1},
$\kappa(G)=1$.

{\em Case $2$.} $\omega(G)=\{1, 3\}$. In this case, $G$ is an
elementary $3$-group of order $3^n$, and by Eq.
(\ref{equation-1}), we have $\kappa(G)=3^{\frac{3^n-1}{2}}$.
Consequently $3^{\frac{3^n-1}{2}}<125$, which forces $n\leqslant
2$.

\begin{itemize}
\item If $n=1$, then $G\cong \mathbb{Z}_3$ and $\kappa(G)=3$.

\item If $n=2$, then $G\cong \mathbb{Z}_3\times\mathbb{Z}_3$ and
$\kappa(G)=81$.
\end{itemize}

{\em Case $3$.} $\omega(G)=\{1, 2, 3\}$. In this case, using a
result of B. H. Neumann \cite{Neumann}, $G=K\rtimes C$ is a
Frobenius group with kernel $K$ and complement $C$, where either
$K\cong \mathbb{Z}_3^t$, $C\cong \mathbb{Z}_2$ or $K\cong
\mathbb{Z}_2^{2t}$, $C\cong \mathbb{Z}_3$.

{\em Subcase $3.1$.} $G=\mathbb{Z}_3^t\rtimes \mathbb{Z}_2$. By
Theorem \ref{th-4} and Eq. (\ref{equation-1}), we have
$$125>\kappa(G)\geqslant \kappa(\mathbb{Z}_3^t)=3^\frac{3^t-1}{2},$$
which implies that $t\leqslant 2$. Thus $G\cong
\mathbb{Z}_3\rtimes \mathbb{Z}_2\cong \mathbb{S}_3$ or
$G\cong(\mathbb{Z}_3\times \mathbb{Z}_3)\rtimes \mathbb{Z}_2$. In
the last case, we see that $\mathbb{Z}_3\times \mathbb{Z}_3$ is
the only Sylow $3$-subgroup of $G$, and thus $G\setminus \{1\}$
has 8 elements of order 3 and 9 element of order 2. Therefore, we
have
$$\mathcal{P}(\mathbb{Z}_3\rtimes \mathbb{Z}_2)=K_1\vee (3K_1\oplus K_2) \ \ \ \mbox{and} \ \ \
\mathcal{P}((\mathbb{Z}_3\times \mathbb{Z}_3)\rtimes
\mathbb{Z}_2)=K_1\vee (9K_1\oplus 4K_2),$$ and so
$\kappa(\mathbb{Z}_3\rtimes \mathbb{Z}_2)=3$ and
$\kappa((\mathbb{Z}_3\times \mathbb{Z}_3)\rtimes
\mathbb{Z}_2)=3^4$.

{\em Subcase $3.2$.} $G=\mathbb{Z}_2^{2t}\rtimes \mathbb{Z}_3$.
Let $x_1, x_2, \ldots, x_m$ be elements of $G$ of order $3$ such
that $\langle x_i\rangle\cap\langle x_j\rangle=\{1\}$. Then, by
Theorem \ref{th-4}, we deduce that
$$125>\kappa(G)\geqslant \prod_{i=1}^{m}\kappa(\langle x_i\rangle)=3^m,$$
which yields that $m\leqslant 4$. This means that, the group $G$
has at most 8 elements of order 3. The only group of this type is
$G=(\mathbb{Z}_2\times\mathbb{Z}_2)\rtimes \mathbb{Z}_3\cong
\mathbb{A}_4$. Clearly, $\mathcal{P}(\mathbb{A}_4)=K_1\vee
(3K_1\oplus 4K_2)$ and we have $\kappa(\mathbb{A}_4)=3^4$.

{\em Case $4$.} $\omega(G)=\{1, 2, 4\}$. In this case $G$ is a
$2$-group with exponent $4$. Suppose that $G$ contains a pair of
elements, say $x_1$ and $x_2$, of order 4 such that $\langle
x_1\rangle\cap\langle x_2\rangle=\{1\}$. Then it follows by
Theorem \ref{th-4} that $\kappa(G)\geqslant \kappa(\langle
x_1\rangle)\cdot \kappa(\langle x_2\rangle)=4^2\cdot 4^2=256$,
which is a contradiction. Now assume that $x_1, x_2, \ldots, x_m$
are elements of order $4$ such that $|\langle
x_i\rangle\cap\langle x_j\rangle|=2$. Then, $\Omega_i=\{1, x_i,
x_i^3\}$, $i=1, 2, \ldots, m-1$, and $\Omega_m=\langle
x_m\rangle$ are cliques in $\mathcal{P}(G)$ with $\Omega_i\cap
\Omega_j=\{1\}$, for each $1\leqslant i<j\leqslant m$, and by
Theorem \ref{th-5}, we deduce that
$$125>\kappa(G)\geqslant \kappa(\langle x_m\rangle)\cdot \prod_{i=1}^{m-1}|\Omega_i|^{|\Omega_i|-2}=4^2\cdot 3^{m-1},$$
which forces $m\leqslant 2$. This means that, the group $G$ has at
most 2 elements of order 4. The only groups $G$ with these
conditions are: $\mathbb{Z}_4$ or $D_8$, and we have
$\kappa(\mathbb{Z}_4)=\kappa(D_8)=2^4$.

{\em Case $5$.} $\omega(G)=\{1, 2, 3, 4\}$. Let $x$ and $y$ be
two elements of order $3$ and $4$, respectively. If there exists
another element $z$ of order 3 such that $\langle x\rangle\cap
\langle z\rangle=1$, then by Theorem \ref{th-4} we have
$$125>\kappa(G)\geqslant \kappa(\langle x\rangle)\cdot \kappa(\langle z\rangle)\cdot \kappa(\langle
y\rangle)=3\cdot 3\cdot 4^2=144,$$ which is a contradiction.
Therefore, the Sylow $3$-subgroup $G_3\in {\rm Syl}_3(G)$ is a
normal subgroup of $G$ of order 3. Moreover $\langle y \rangle$
acts fixed-point-freely by conjugation on $G_3$, and so
$G_3:\langle y\rangle$ is a Frobenius group. But then, we must
have $|\langle y\rangle|\big | |G_3|-1$, which is impossible.
Thus, in this case there does not exist a candidate for $G$. This
completes the proof. \end{proof}

The next result is a simple consequence of Theorem \ref{th-small}.
\begin{corollary}\label{coro-small} The following statements hold:
\begin{itemize}
\item[{\rm (a)}]
$h_{\mathcal{F}}(\mathbb{Z}_3)=h_{\mathcal{F}}(\mathbb{S}_3)=2$.

\item[{\rm (b)}] $h_{\mathcal{F}}(\mathbb{Z}_3\times \mathbb{Z}_3
)=h_{\mathcal{F}}((\mathbb{Z}_3\times \mathbb{Z}_3)\rtimes
\mathbb{Z}_2)=h_{\mathcal{F}}(\mathbb{A}_4)=3$.

\item[{\rm (c)}] $h_{\mathcal{F}}(\mathbb{Z}_4)=h_{\mathcal{F}}(D_8)=2$.
\end{itemize}
\end{corollary}

\subsection{A New Characterization of $\mathbb{A}_5$}
The following result is needed to prove Theorem \ref{th1-new}.
\begin{lm}\label{marty} {\rm ({\bf I. M. Isaacs})}
Let $G$ be a finite non-abelian simple group and let $p$ be a
prime dividing the order of $G$. Then $G$ has at least $p^2-1$
elements of order $p$, or equivalently, there is at least $p+1$
cyclic subgroups of order $p$ in $G$.
\end{lm}
\begin{proof}
Let $U$ be a subgroup of order $p$ contained in the center of a
Sylow $p$-subgroup of $G$, and let $N = N_G(U)$. Since $G$ is
non-abelian simple, $U$ is not normal in $G$ and thus $N<G$. Write
$n=|G:N|$. By the elementary ``$n$ factorial theorem" we know
that $|G:{\rm core}_G(N)|$ divides $n!$, and since $G$ is simple
and ${\rm core}_G(N)\leqslant N<G$, we have ${\rm core}_G(N)=1$.
Thus $|G|$, and so $p$ divides $n!$, which forces $n\geqslant p$.
But $N$ contains a full Sylow $p$-subgroup of $G$ so $n$ is not
$p$, and thus $n\geqslant p+1$. Then $G$ contains at least $p+1$
conjugates of $U$.
\end{proof}

{\em Proof of Theorem \ref{th1-new}.}  Let $G$ be a finite simple
group with $\kappa(G)=\kappa(\mathbb{A}_5)$. First of all, we
recall, $\mathbb{A}_5\setminus \{1\}$ has $24$ elements of order
$5$, $20$ elements of order $3$, and $15$ elements of order $2$,
and hence $\omega(\mathbb{A}_5)=\{1, 2, 3, 5\}$. Thus, the power
graph $\mathcal{P}(\mathbb{A}_5)$ is the union of complete
subgraphs $K_5$, $K_3$ and $K_2$ with exactly one common vertex,
i.e., the identity element $1$, or equivalently
$$\mathcal{P}(\mathbb{A}_5)=K_1\vee (15K_1\oplus 10K_2\oplus 6K_4).$$
Now, by Corollary \ref{cor-epo}, we have
$$k(\mathbb{A}_5)=\kappa(K_2)^{15}\kappa(K_3)^{10}\kappa(K_5)^6=(2^0)^{15}(3^1)^{10}(5^3)^6=3^{10}5^{18}.$$
Therefore, $G$ is a finite simple group with
$\kappa(G)=3^{10}5^{18}$. Clearly, $G$ is non-abelian, since
otherwise $G\cong \mathbb{Z}_p$ for some prime $p$, and so
$\kappa(G)=\kappa(\mathbb{Z}_p)=p^{p-2}$, which is a
contradiction. Now, we claim that $\pi(G)=\{2, 3, 5\}$. Since
$\kappa(G)=3^{10}5^{18}<17^{15}$, from Corollary
\ref{cor-semidir-1}, it follows that  $$\pi(G)\subseteq
\pi(16!)=\{2, 3, 5, 7, 11, 13\}.$$ Suppose now that $p\in \pi(G)$
and $p\geqslant 7$. Let $s_p$ be the number of elements of order
$p$. Then, $s_p= c_p\phi(p)=c_p(p-1)$, where $c_p$ is the number
of cyclic subgroups of order $p$ in $G$. By Lemma \ref{marty},
$c_p\geqslant p+1$, because $G$ is a non-abelian simple group.
Therefore, from Theorem \ref{th-4} $(a)$, we deduce that
$$\kappa(G)\geqslant
\kappa(\mathbb{Z}_p)^{p+1}=p^{(p-2)(p+1)}\geqslant
7^{40}>3^{10}5^{18},$$ which is a contradiction. Our claim
follows.

By results collected in  \cite[Table 1]{zav}, $G$ is isomorphic
to one of the groups $\mathbb{A}_5$, $\mathbb{A}_6$ or $U_4(2)$.
In all cases, $G$ does not contain an element of order 15 (see
\cite{atlas}). Hence, a Sylow $3$-subgroup $G_3$ acts fixed point
freely on the set of elements of order 5 by conjugation, and so
$|G_3|$ divides $s_5=c_5\phi(5)=4c_5$. Now, if $|G_3|\geqslant
9$, then $c_5\geqslant 9$, and so
$$\kappa(G)>\kappa(\mathbb{Z}_5)^9=(5^3)^9>3^{10}5^{18},$$
which is a contradiction. Therefore, $|G_3|=3$, which forces
$G\cong \mathbb{A}_5$. This completes the proof. $\Box$

\begin{center} {\bf Acknowledgments}
\end{center} We sincerely thank the referee for an exceptionally careful
reading of the paper and for several valuable suggestions. We
would like to thank professor I. M. Isaacs for his proof of Lemma
\ref{marty}.

\subsection{Appendix} Using the obtained results in
previous sections we can compute the tree-numbers $\kappa(G)$ and
$\kappa(G^\#)$ of a finite group $G$, with $|G|\leqslant 15$.
These results are shown below in Table 1. In this table, $M$
denotes a non-abelian group of order 12 with the following
presentation:
$$M=\langle x, y \ | \ x^4=y^3=1, yx=xy^2\rangle.$$
\begin{center}
{\bf Table 1.} \ {\em The tree-numbers of small finite
groups}\\[0.3cm]
$
\begin{array}{llll}
\hline
n & \mbox{Group of order} \ n & \kappa(G) & \kappa(G^\#) \\
\hline 1 & 1 & 1 & - \\
2 & \mathbb{Z}_2 & 1 & 1 \\
3 & \mathbb{Z}_3 & 3 & 1 \\
4 & \mathbb{Z}_4 & 2^4 & 3 \\
& \mathbb{Z}_2\times \mathbb{Z}_2 & 1 & 0 \\
5 & \mathbb{Z}_5 & 5^3 & 2^4 \\
6 & \mathbb{Z}_6\cong \mathbb{Z}_3\times \mathbb{Z}_2 &  2^2\cdot
3^3\cdot 5 & 2^3\cdot 5 \\
& \mathbb{S}_3\cong \mathbb{Z}_3\rtimes \mathbb{Z}_2 & 3 & 0 \\
7 & \mathbb{Z}_7 & 7^5 & 2^4\cdot 3^4\\
8 & \mathbb{Z}_8 & 2^{18} & 7^5 \\
& \mathbb{Z}_4\times \mathbb{Z}_2 & 2^6\cdot 3 & 0 \\
& \mathbb{Z}_2\times \mathbb{Z}_2\times \mathbb{Z}_2 & 1 & 0
\\
& D_8 & 2^4 & 0 \\
& Q_8 & 2^{11} & 3^3 \\
9 & \mathbb{Z}_9 & 3^{14} & 2^{18} \\
& \mathbb{Z}_3\times \mathbb{Z}_3 & 3^4 & 0 \\
10 & \mathbb{Z}_{10}\cong\mathbb{Z}_{5}\times \mathbb{Z}_2  &
2^4\cdot 3^6\cdot 5^5 & 2^{11}\cdot 3^6 \\
& D_{10}\cong \mathbb{Z}_5\rtimes \mathbb{Z}_2 & 5^3 & 0 \\
11 & \mathbb{Z}_{11} & 11^9 & 2^8\cdot 5^8\\
12 & \mathbb{Z}_{12}\cong\mathbb{Z}_{4}\times \mathbb{Z}_3  &
2^{14}\cdot 3^6\cdot 5\cdot 131 &  2^{4}\cdot 3^2\cdot 7\cdot 11^3\cdot 173 \\
& \mathbb{Z}_6\times \mathbb{Z}_2 & 2^6\cdot 3^5\cdot 5^2\cdot 17 & 2^8\cdot 5^3 \\
& \mathbb{A}_4\cong(\mathbb{Z}_2\times \mathbb{Z}_2)\rtimes \mathbb{Z}_3 & 3^4 & 0 \\
& D_{12}\cong \mathbb{Z}_6\rtimes \mathbb{Z}_2 & 2^2\cdot
3^3\cdot 5 & 0 \\
& M & 2^{11}\cdot 3^2\cdot 5\cdot 7 & 2^3\cdot 3^3\cdot 5 \\
13 & \mathbb{Z}_{13} & 13^{11} & 2^{20}\cdot 3^{10} \\
14 & \mathbb{Z}_{14}\cong\mathbb{Z}_{7}\times \mathbb{Z}_2  & 2^6\cdot 7^7\cdot 13^5 &  2^{11}\cdot 3^6\cdot 13^5 \\
& D_{14}\cong \mathbb{Z}_7\rtimes \mathbb{Z}_2 & 7^5 & 0\\
15 & \mathbb{Z}_{15}\cong\mathbb{Z}_{5}\times \mathbb{Z}_3  &
3^{10}\cdot 5^{8}\cdot 11 \cdot 13^3  &
2^{17}\cdot 3^{3}\cdot 5\cdot 7^7 \\
\hline
\end{array}$
\end{center}

\end{document}